\documentclass[11pt]{article}
\usepackage{amsmath}
\usepackage{dsfont}
\usepackage{mathrsfs}
\usepackage{amsmath,amssymb}
\usepackage{amsfonts}
\usepackage{hyperref}
\usepackage{amsthm}
\usepackage{graphicx}
\usepackage{subfigure}
\usepackage{xcolor}

\usepackage{bbm}
\usepackage{mathrsfs}
\usepackage{txfonts}
\usepackage{color}
\usepackage{pict2e}
\usepackage{palatino,epsfig,latexsym}
\usepackage{epsf,xy,epic,amscd}

\theoremstyle{plain}
\newtheorem{theorem}{Theorem}[section]
\newtheorem{lemma}[theorem]{Lemma}

\newtheorem{proposition}[theorem]{Proposition}
\newtheorem{corollary}[theorem]{Corollary}

\newtheorem{problem}[theorem]{Problem}
\newtheorem{Bounded Diameter Lemma}[theorem]{Bounded Diameter Lemma}
\theoremstyle{definition}
\newtheorem{definition}[theorem]{Definition}

\newtheorem{remark}[theorem]{Remark}

\newcommand{\Hmm}[1]{\leavevmode{\marginpar{\tiny%
$\hbox to 0mm{\hspace*{-0.5mm}$\leftarrow$\hss}%
\vcenter{\vrule depth 0.1mm height 0.1mm width \the\marginparwidth}%
\hbox to
0mm{\hss$\rightarrow$\hspace*{-0.5mm}}$\\\relax\raggedright #1}}}

\hfuzz=\maxdimen
\DeclareFixedFont{\Acknowledgment}{OT1}{cmr}{bx}{n}{14pt}
\newtheorem{df}{Definition}[section]

\newtheorem{conj}[df]{Conjecture}

\newcommand{\pf}{\textit{Proof.} }

\newcommand{\ti}{\tilde}

\newcommand{\pr}{\prime}
\newcommand{\ep}{\epsilon}
\newcommand{\del}{\delta}
\newcommand{\mc}{\mathcal}

\newcommand{\p}{\partial}
\newcommand{\eq}{eqnarray*}
\begin{document}
\title{Rigidity of the hexagonal Delaunay triangulated plane}
\author{Song Dai, Huabin Ge, Shiguang Ma}
\date{}

\maketitle

\begin{abstract}
We show the rigidity of the hexagonal Delaunay triangulated plane under Luo's PL conformality.
As a consequence, we obtain a rigidity theorem for a particular type of locally finite convex ideal hyperbolic polyhedra. 

\end{abstract}



\section{Introduction}\label{intro}
Let $\Sigma$ be a surface without boundary, and $T=(V,E,F)$ be a triangulation on $\Sigma$, where $V$ is the set of vertices, $E$ is the set of edges and $F$ is the set of faces. For a function $f$ on $V$, $i\in V$, we sometimes write $f_i$ instead of $f(i)$. We also use this notation for the function on $E,F$. A piecewise linear metric (PL  metric for short) is a function $l: E\rightarrow \mathbb{R}_{>0}$, such that for each $ijk\in F$, $ijk$ forms a Euclidean triangle. If we require the lengths of the edges are positive but the triangle inequality can be equality, we call $l$ a generalized triangle. Given a generalized PL metric $l$, it induces an intrinsic distance and a \emph{flat cone metric} on the triangulation $T$ in the natural manner. For each vertex $i\in V$, it is a cone point with singularity expressed as the discrete Gaussian curvature $K_i$, defined by
\begin{equation*}
K_i=2\pi-\sum\limits_{kij\in F}\theta_{\angle kij},
\end{equation*}
where $\theta_{\angle kij}$ is the angle of $\angle kij$. A generalized PL metric $l$ is called \emph{flat}, if there is no singularity, that is, $K_i=0$ at each vertex $i\in V$. In \cite{Luo}, Luo introduced the notion of the PL conformality.
\begin{definition}[\cite{Luo}]\label{plc}
Let $l,\ti{l}$ be two generalized PL metrics on $(\Sigma,T)$. We call $l$ and $\ti{l}$ are PL conformal if
\begin{equation*}
\ti{l}_{ij}=e^{u_i+u_j}l_{ij},\quad \forall ij\in E
\end{equation*}
for some function $u: V\rightarrow \mathbb{R}$. Denote $\ti{l}=u*l$.
\end{definition}
The function $u$ is called a PL conformal factor, which plays an analogous role as in the smooth case. Motivated by the prescribed curvature problem in the smooth case, Luo proposed the combinatorial version of the prescribed curvature problem in the settings above. (If the prescribed curvature is constant, it is the Yamabe problem.)

\begin{problem}[\cite{Luo}]
\label{question-luo}
Let $\Sigma$ be a surface (without boundary) with a triangulation $T$. Given a PL metric $l_0$ and a prescribed curvature $K:V\rightarrow \mathbb{R}$, is there a PL metric $l$ that PL conformal to $l_0$ and has discrete Gaussian curvature $K$? Is it unique if it exists? (rigidity)
\end{problem}

In case $\Sigma$ is compact, Problem \ref{question-luo} was perfectly resolved. For the rigidity part, Luo \cite{Luo} first proved a local version and conjectured that the global rigidity still holds. Using a variational principle and an extension technique, Bobenko, Pinkall and Springborn \cite{BPS} affirmatively answered Luo's global rigidity conjecture (see \cite{GeJiang-CVPDE} for further development). They further equipped the triangulated PL surface $(\Sigma, T, l)$ with a canonical hyperbolic metric with cusps, and observed that two PL metrics (with the same triangulation) are PL conformal if and only if the corresponding hyperbolic metrics are isometric. Due to this observation, they generalized Definition \ref{plc} to PL metrics that may not be combinatorially equivalent (see Definition 5.1.4 in \cite{BPS}, and Definition 1.1 in \cite{GLSW} for an equivalent but more algorithmic definition). Under this viewpoint, Gu, Luo, Sun and Wu \cite{GLSW} obtained a discrete uniformization theorem for PL metrics with the help of the decorated Teichm\"uller space theory. Their discrete uniformization theorem completely resolved the existence part of Problem \ref{question-luo} (under Bobenko, Pinkall and Springborn's definition of discrete conformality). Similarly, a hyperbolic version of the discrete uniformization theory was established in \cite{GGLSW}. It is remarkable that Springborn \cite{S} established the equivalence between the discrete uniformization theorem on $\mathbb{S}^2$ and Rivin's realization theorem for ideal hyperbolic polyhedra \cite{Rivin}.

In case $\Sigma$ is non-compact, very little results are known related to Problem \ref{question-luo}. Inspired by Rodin and Sullivan's celebrated work \cite{Rodin-Sullivan}, where they proved Thurston's conjecture (i.e. the only complete flat circle packing metric on the hexagonal triangulation of the plane is the regular hexagonal packing), Wu, Gu and Sun \cite{WGS} considered the rigidity problem for the infinite hexagonal triangulation $T$ of the plane $\Sigma=\mathbb{C}$, see Figure \ref{fig-10}.

\begin{theorem}(\cite{WGS})
Let $l$ be a PL metric on $\mathbb{C}$ with standard hexagonal triangulation, which is PL conformal to $l_0\equiv1$. Suppose $(T,l)$ is flat, complete (i.e. isometric to $(T,l_0)$) and there is a constant $\del>0$ such that all angles $\leq \frac{\pi}{2}-\del$ ($\delta$-condition). Then $l \equiv C$ for some constant $C>0$.
\end{theorem}

The \emph{$\delta$-condition} appeared above, while suitable for some purposes, is considerably less satisfying. The main result of this paper is to release the \emph{$\delta$-condition} to the \emph{Delaunay condition}. For each edge $ij\in E$, consider the two adjacent triangles $\triangle ijk$ and $\triangle ijl$, denote the sum of the opposite angles $\alpha_{ij}=\theta_{\angle ikj}+\theta_{\angle ilj}$, see Figure \ref{fig-g1}. A generalized PL metric on the triangulated surface $(T,l)$ is called \emph{Delaunay} if $\alpha_{ij}\leq \pi$ for all $ij\in E$. We have

\begin{theorem}\label{main}
Let $l$ be a generalized PL metric on $\mathbb{C}$ with standard hexagonal triangulation, which is PL conformal to $l_0\equiv1$. Suppose $(T,l)$ is flat, complete (i.e. isometric to $(T,l_0)$) and is Delaunay. Then $l\equiv C$ for some constant $C>0$.
\end{theorem}

\begin{figure}
\begin{minipage}[t]{0.48\linewidth}
\centering
\includegraphics[width=0.8\textwidth]{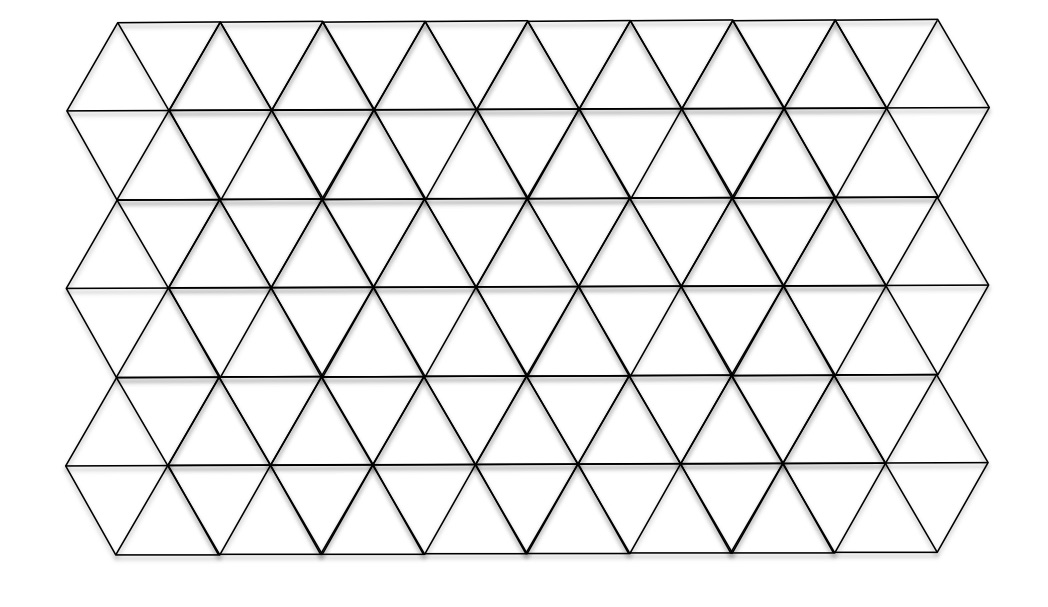}
\caption{A hexagonal triangulation}\label{fig-10}
\end{minipage}
\begin{minipage}[t]{0.5\linewidth}
\centering
\includegraphics[width=0.66\textwidth]{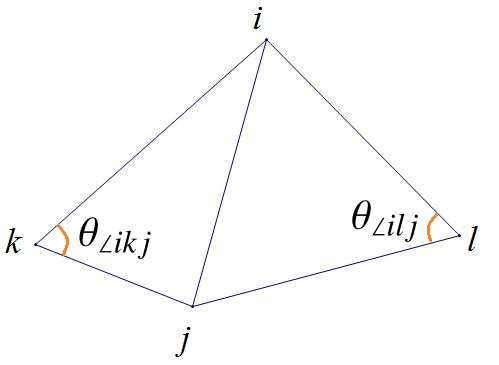}
\caption{An interior edge}\label{fig-g1}
\end{minipage}
\end{figure}

The Delaunay condition in Theorem \ref{main} is relatively satisfying. We shall show (see Section \ref{hyp}) that a PL metric is Delaunay if and only if the corresponding ideal hyperbolic polyhedron is convex. Moreover, the rigidity of PL conformality in Theorem \ref{main} is equivalent to a rigidity result for ideal hyperbolic polyhedra, which may be considered as an infinite and hyperbolic version of Cauchy \cite{Cauchy} and Alexandrov's \cite{Alexandrov}\cite{Bobenko-Izmestiev} rigidity for Euclidean polyhedra. The Delaunay condition is a satisfying condition in the sense that, generally, polyhedron rigidity holds only for the convex ones. See Section \ref{hyp} for more details.


The condition of completeness in Theorem \ref{main} can't be removed. In fact, we consider $V$ as the points in $\mathbb{C}$ with complex coordinate $m\cdot 1+n\cdot \omega$ with $\omega=\frac{1+\sqrt{-3}}{2}$ for $m,n\in \mathbb{Z}$. Let $u$ be the restriction on $V$ of a linear function $\ti{u}=az+b$. Then by similarity we see $u*l_0$ is flat. For suitable $a$, the picture is shown in Figure \ref{fig-9}, which is regarded as a lift of the covering map $\mathbb{C}\rightarrow \mathbb{C}\setminus \{0\}$. The PL metric in this way is not complete unless $u$ is constant. Wu, Gu and Sun \cite{WGS} conjectured that it is the only possibility when just assuming flatness.

\begin{figure}
\centering
\includegraphics[width=0.4\textwidth]{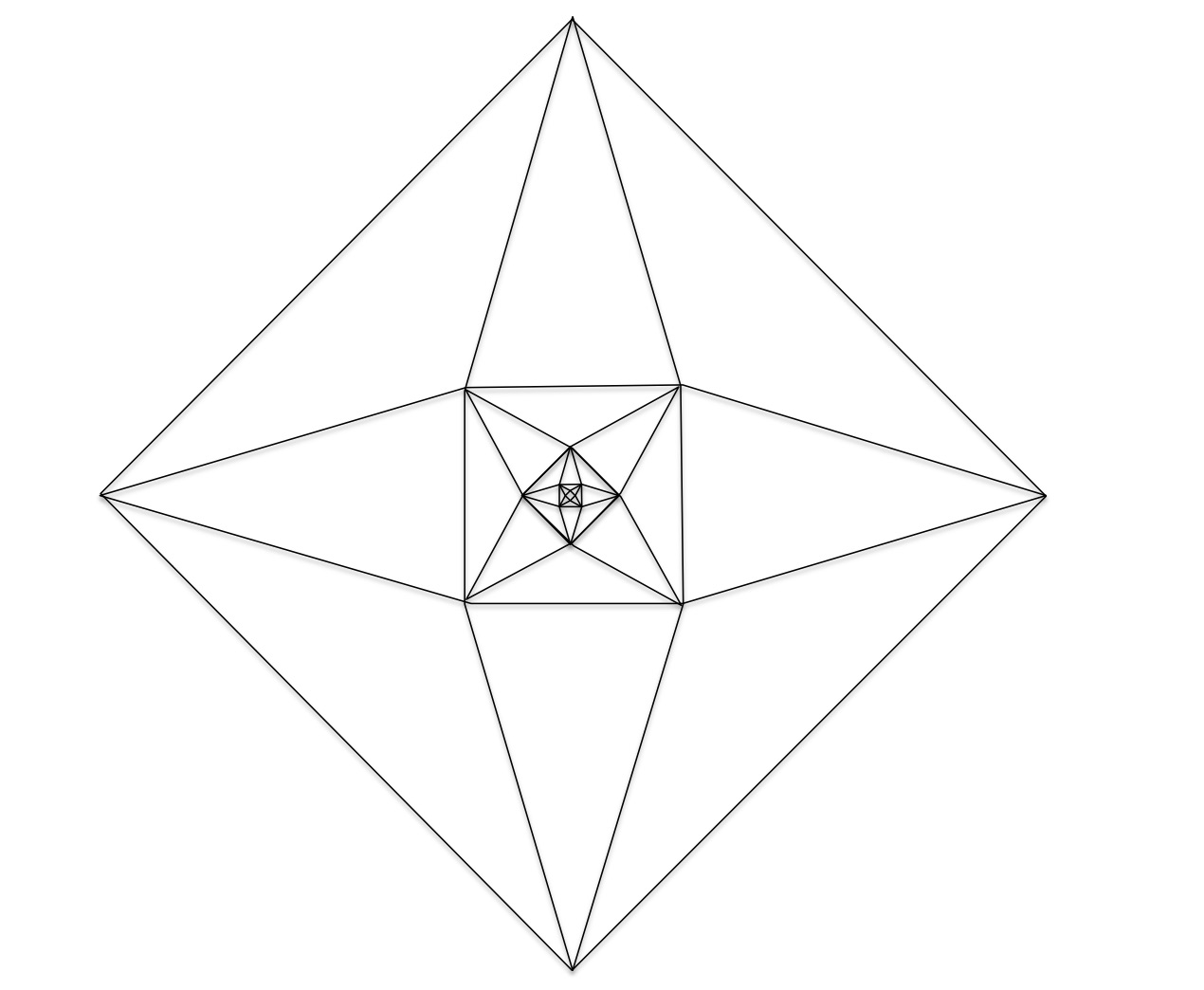}
\caption{A noncomplete hexagonal triangulation}\label{fig-9}
\end{figure}

\begin{conj}
Let $l$ be a PL metric on $\mathbb{C}$ with standard hexagonal triangulation, which is PL conformal to $l_0\equiv 1$. Suppose $(T,l)$ is flat. Then the conformal factor $u$ is the restriction on $V$ of a linear function $\ti{u}=az+b$.
\end{conj}

The paper is organized as follows. We prove the main Theorem \ref{main} in Section \ref{p3}. In Section \ref{sec-outline-proof}, we outline the proof of Theorem \ref{main}. The main technical Lemma \ref{mp} and Proposition \ref{2.2} are postponed to Section \ref{p1} and Section \ref{p2} respectively. In Section \ref{hyp}, we interpret Theorem \ref{main} from the viewpoint of hyperbolic geometry, to a rigidity theorem for ideal convex polyhedra.\\

\noindent\textbf{Acknowledgement:}
During the preparation of this article, we showed our results to Professor Feng Luo. He told us that he and his coauthors had already got similar rigidity results (but without posting them online), which had been submitted to a journal. We thank Professor Feng Luo and Tianqi Wu for helpful conversations and suggestions. We would also like to thank Xiaoxiao Zhang for drawing the pictures in the paper. The first author is supported by NSF of China (No.11871283 and No.11971244). The second author is supported by NSF of China (No.11871094). The second author would like to thank the hospitality of Professor Weiping Zhang and Huitao Feng during his visit to Chern Institute of Mathematics in Spring 2018 when he initiated this collaboration. The third author is supported by NSF of China (No.11571185 and No.11871283), China Scholarship Council (No. 201706135016), and the Fundamental Research Funds for the Central Universities, Nankai University(63191506).

\section{Proof of Theorem \ref{main}}\label{p3}

\subsection{Outline of the proof}\label{sec-outline-proof}
In this section, we outline the proof of Theorem \ref{main}. The key step is to establish a maximum principle in the PL conformal settings. Let $H=H(i_0;i_1,i_2,i_3,i_4,i_5,i_6)$ be a hexagon centered at $i_0$ with PL metric $l_0\equiv 1$, see Figure \ref{fig-8}.

Let $u$ be a conformal factor and $l=u*l_0$. If there is no confusion, we use $u_j$ instead of $u_{i_j}$. Since the angles are invariant under the similar transformation, sometimes we assume $u_0=0$. For $u_0=0$, the length of the edge $i_ai_b$ is $l_{ab}=e^{u_a+u_b}$. Denote $\mc{T}$ as the space of the conformal factors such that the corresponding hexagons consisting of six generalized triangles. Recall for a generalized triangle, we mean the lengths of the edges are positive but the triangle inequality can be equality.
\begin{equation*}
\mc{T}=\{u=(u_0,u_1,\cdots,u_6)\in \mathbb{R}^7: ~ l_{ab}+l_{bc}\geq l_{ca}\text{ for } \{a,b,c\}=\{i_0,i_j,i_{j+1}\},j=1,\cdots,6\}.
\end{equation*}
Throughout this paper, the index $i_{k+6}=i_k$ for $k\geq 1$. For a generalized triangle, the angles are well defined and are continuously extended to the degenerate case. (If the equality of triangle inequality holds for a triangle, then we call the triangle is degenerate.) For example, for $abc\in F$, if $l_{ab}=l_{bc}+l_{ca}$, then $\theta_{\angle acb}=\pi$, $\theta_{\angle cab}=\theta_{\angle abc}=0$.
For $u\in\mc{T}$, denote $$\theta_{u}=\sum\limits_{j=1}^{6}\theta_{\angle i_ji_0i_{j+1}},\quad K_{u}=2\pi-\theta_{u}.$$

\begin{figure}
\centering
\includegraphics[width=0.4\textwidth]{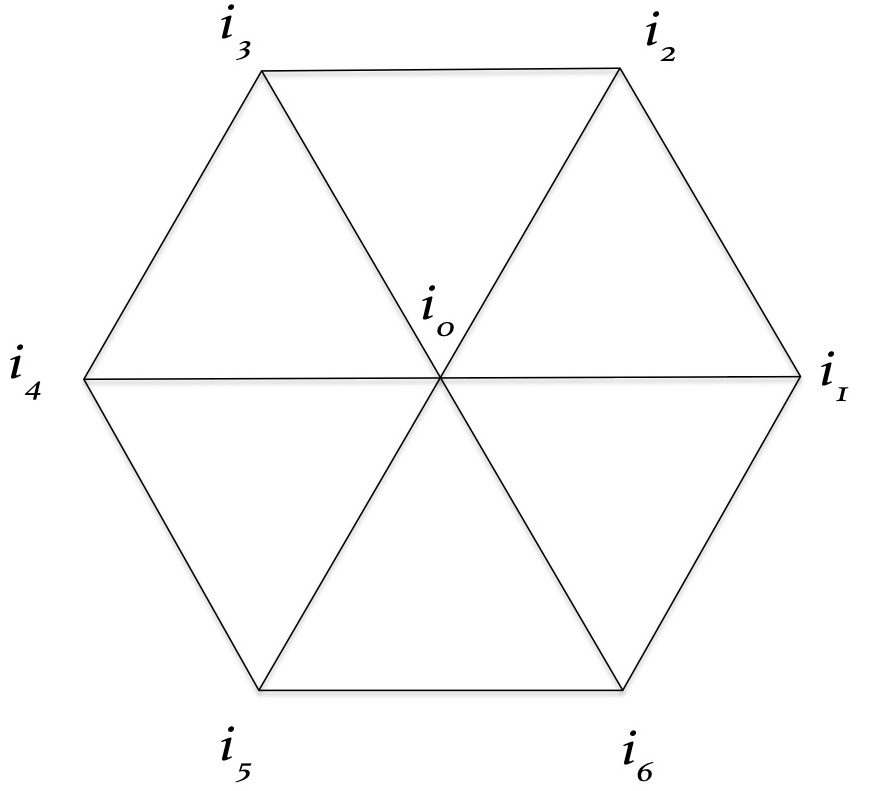}
\caption{A hexagon with center}\label{fig-8}
\end{figure}

Denote $\mc{D}$ as the space of the conformal factors such that the corresponding hexagons in $\mc{T}$ are Delaunay. Recall $\alpha_{i_0i_j}=\theta_{\angle i_0i_{j-1}i_j}+\theta_{\angle i_0i_{j+1}i_j}$.
$$\mc{D}=\{u\in\mc{T}: ~\alpha_{i_0i_{j}}\leq \pi,~j=1,\cdots,6\}.$$
For $u,v\in\mathbb{R}^7$, we denote $u\geq v$ if $u_i\geq v_i$ for $i=1,\cdots, 7.$ The maximum principle is:
\begin{lemma}\label{mp}
Let $\bar{u},\underline{u}\in \mc{D}$ with $\bar{u}_0=\underline{u}_0=0$. If $K_{\underline{u}}=K_{\bar{u}}=0$ and $\bar{u}\geq \underline{u}$, then $\bar{u}=\underline{u}$.
\end{lemma}
For more general applications, we prove the maximum principle for general $n$ in Lemma \ref{key}.
\begin{remark}
The maximum principle above has the same formation as the strong maximum principle in the smooth case. That is, let $D$ be a bounded domain, $p\in D$, suppose $\triangle \bar{u}=\triangle \underline{u}=0$ in $D$ and $\bar{u}\geq \underline{u}$ on $\partial D$, then $\bar{u}(p)= \underline{u}(p)$ implies $\bar{u}\equiv \underline{u}$ in $D$.
\end{remark}
\begin{remark}
From Corollary \ref{derivative}, for $u\in \mc{D}$, we have $\frac{\partial K_u}{\partial u_j}\leq 0$. It seems we can directly obtain $K_{\underline{u}}\geq K_{\bar{u}}$. But the problem is in $\mc{D}$, apriori we don't know whether we can connect $\bar{u}$ and $\underline{u}$ by a broken line such that every segment is along the positive direction of a coordinate axis. In fact, it is what we show in the proof.
\end{remark}

Let $(T,l)$ be a hexagonal triangulated plane with generalized PL metric $l$ conformal to the standard one. Then the vertices of $T$ in $\mathbb{C}$ are $\mathbb{Z}$-spanned by $1$ and $\omega=\frac{1+\sqrt{-3}}{2}$. Let $l=u*l_0$, $\nabla_c u(i)=u(i+c)-u(i)$ for $c=1$ or $\omega$. For $i\in V$, $R\in \mathbb{N}$, denote the ball $B(i,R)$ as the set of triangles whose vertices can be connected with $i$ by a path of at most $R$ edges.

From the maximum principle above, we can show the following proposition, which plays a similar role as Wu-Gu-Sun's Lemma 2.2 in \cite{WGS}. 
\begin{proposition}\label{2.2}
For any $\ep>0$, $R\in \mathbb{N}$, there exists a constant $\delta>0$ depending on $\ep,R$, such that for any $M>0$, $i\in V$, if
$$\nabla_c u(i)\geq M-\delta \text{ and } \nabla_c u|_{B(i,R)}\leq M,$$
then $\nabla_c u|_{B(i,R)}\geq M-\ep$.
\end{proposition}

We postpone the proof of Lemma \ref{mp} to Section \ref{p1} and prove it for general $n$. We postpone the proof of Proposition \ref{2.2} to Section  \ref{p2}.
Under Proposition \ref{2.2}, Wu-Gu-Sun's method in \cite{WGS} can apply to prove Theorem \ref{main}. Notice that the notation $\Delta$ in \cite{WGS} is replaced by $\nabla$ in this paper.

First we show the ``gradient" $\nabla u$ has a universal bound apriori, which is also showed in Lemma 2.4 of \cite{WGS}.
\begin{lemma}\label{ratio}
Let $i_0i_1i_2$ form a triangle in $T$. Then there is a universal constant $\ep_0$, such that $l_{i_1i_2}\geq \ep_0 l_{i_0i_2}$. In other words, there a universal constant $M_0$ such that $|\nabla u|\leq M_0$, where $\nabla u=u(i^{\pr})-u(i)$ for some $i\sim i^{\pr}$.
\end{lemma}
\begin{proof}
Consider the hexagon $H=(i_0;i_1,i_2,i_3,i_4,i_5,i_6)$ centered at $i_0$ as in Figure \ref{fig-8}. Since the length ratio is invariant under the similarity. We may assume $u_0=0$. In this case, $l_{i_ji_{j+1}}=l_{i_0i_{j}}l_{i_0i_{j+1}}$. We show $l_{i_0i_{1}}$ has a universal positive lower bound. If not, for any $\ep_1\leq \frac{1}{2}$, suppose $l_{i_0i_{1}}\leq \ep_1$, then from $l_{i_0i_{2}}\leq l_{i_0i_{1}}+l_{i_0i_{1}}l_{i_0i_{2}}$, we obtain $l_{i_0i_{2}}\leq \frac{l_{i_0i_{1}}}{1-l_{i_0i_{1}}}\leq 2\ep_1$. Repeat this procedure, we may assume for any $\ep$, there is a flat PL metric on $H$ conformal to $l_0$ such that $u_0=0$ and $l_{i_0i_j}\leq \ep$ for $j=1,\cdots,6$. From the triangle inequality, we have $$\frac{l_{i_0i_{j}}}{1+l_{i_0i_{j}}}\leq l_{i_0i_{j+1}}\leq \frac{l_{i_0i_{j}}}{1-l_{i_0i_{j}}}.$$
To consider $\theta_{\angle i_ji_0i_{j+1}}$, we have
\begin{\eq}
\cos\theta_{\angle i_ji_0i_{j+1}}&=&\frac{l_{i_0i_{j}}^2+l_{i_0i_{j+1}}^2-l_{i_0i_{j}}^2l_{i_0i_{j+1}}^2}{2l_{i_0i_{j}}l_{i_0i_{j+1}}}\\
&\geq& \frac{l_{i_0i_{j}}^2+\frac{l_{i_0i_{j}}^2}{(1+l_{i_0i_{j}})^2}}{2l_{i_0i_{j}}\frac{l_{i_0i_{j}}}{1-l_{i_0i_{j}}}}
-\frac{l_{i_0i_{j}}l_{i_0i_{j+1}}}{2}\\
&=&\frac{1+\frac{1}{(1+l_{i_0i_{j}})^2}}{2}(1-l_{i_0i_{j}})-\frac{l_{i_0i_{j}}l_{i_0i_{j+1}}}{2}.
\end{\eq}
So when $l_{i_0i_{j}},l_{i_0i_{j+1}}$ approach to $0$, $\theta_{\angle i_ji_0i_{j+1}}$ approaches to $0$. It contradicts to the assumption that the PL metric is flat. Notice that $u_0=0$, so $l_{i_0i_1}=e^{u_1-u_0}$ has a universal positive lower bound implies $u_1-u_0$ has a universal lower bound. Since $\triangle i_0i_1i_2$ is arbitrary, we obtain $\nabla u$ has a universal bound.
\end{proof}

When $\nabla_1 u$ and $\nabla_{\omega}u$ are constants, one can show
\begin{lemma}\label{overlap1}
For $M,N\in \mathbb{R}$, $(M,N)\neq (0,0)$, if $\nabla_1 u\equiv M$ and $\nabla_{\omega} u\equiv N$, then there is a constant $R(M,N)$ depending on $M,N$, such that for any $i\in V$, the ball $B(i,R(M,N))$ must have an overlap of positive area.
\end{lemma}
\begin{proof} It is proved in Lemma 2.6 of \cite{WGS}.
\end{proof}

By perturbation, we show
\begin{lemma}\label{overlap2}
For any $M,N\in \mathbb{R}$, $M\neq 0$, there exists $\ep(M)>0$ and $R(M)\in \mathbb{N}$, such that for any $i\in V$, if
$$|\nabla_1 u-M|< \ep(M),~ |\nabla_{\omega} u-N|<\ep(M) \text{ in } B(i,R(M)),$$ then the ball $B(i,R(M))$ must have an overlap of area.
\end{lemma}
\pf Given $(M,N)$, $M\neq 0$, choose $R(M,N)$ as in Lemma \ref{overlap1}. For $i\in V$, since the overlap of positive area is an open condition, then there exists $\ep_1=\ep_1(M,N,i)$, such that $B(i,R(M,N))$ has an overlap of positive area. When $\nabla_1 u\equiv M,~\nabla_{\omega} u\equiv N$, all balls of radius $R$ are similar, so $\ep_1$ only depends on $M,N$. So Lemma \ref{overlap2} holds for $\ep=\ep_1$.
Fix $M$, since the set $S:=\{N:$ the configuration $\nabla_1 u\equiv M, \nabla_{\omega} u\equiv N$ consists of generalized triangles. $\}$ is a compact set. So we can choose the constant $\ep$ only depending on $M$.
\qed

Finally, by showing the following lemma, together with Lemma \ref{overlap2}, we finish the proof of Theorem \ref{main}. Under Proposition \ref{2.2}, the proof of Lemma \ref{MN} is basically similar to the proof of Lemma 2.3 in \cite{WGS}. For the convenience of the readers, we also give a proof here. From Lemma \ref{ratio}, $|\nabla_c u|$ are bounded, $c=1$ or $\omega$.
\begin{lemma}\label{MN}
Let $M=\sup \nabla_1 u$. For any $\ep>0$, $R\in\mathbb{N}$, there exists $i=i(M,\ep,R)\in V$ and $N=N(M,\ep,R)\in \mathbb{R}$ depending on $M,\ep,R$, such that $$|\nabla_1 u-M|< \ep,~ |\nabla_{\omega} u-N|<\ep \text{ in } B(i,R).$$
\end{lemma}
\pf Choose $\delta=\delta(\ep,R)$ as in Proposition \ref{2.2}. Let $n=[\frac{2M_0}{\delta}]+1$. Let $R_1=nR$. Choose $\delta_1=\delta_1(\ep,R_1)$ as in Proposition \ref{2.2}. Since $M=\sup \nabla_1 u$, we can choose $i_0\in V$ such that $\nabla_1 u(i_0)\geq M-\delta_1$, which implies $\nabla_1 u|_{B(i,R_1)}\geq M-\epsilon$ from Proposition \ref{2.2}.

From Lemma \ref{ratio}, we suppose $|\nabla_{\omega}u|\leq M_0$.  Let $F(k)$ be the maximum of $\nabla_{\omega}u$ in $B(i_0,kR)$ for $k=0,1,\cdots,n$. Then $$-M_0\leq F(0)\leq F(1)\leq\cdots\leq F(n)\leq M_0.$$ Hence there exists $k_0\in\{1,\cdots,n\}$ such that $$F(k_0)-F(k_0-1)\leq \frac{2M_0}{n}\leq \delta.$$ Suppose $\nabla_{\omega}u(j_0)=F(k_0-1)$ for some $j_0\in B\big(i_0,(k_0-1)R\big)$. Let $N=F(k_0)$, $i=j_0$. Then $\nabla_{\omega}u(i)\geq N-\delta$. Since $B(i,R)\subseteq B(j_0,k_0 R)$, we have $\nabla_{\omega}u|_{B(i,R)}\leq N$. Applying Proposition \ref{2.2}, we obtain $\nabla_{\omega}u|_{B(i,R)}\geq N-\ep$. So $|\nabla_{\omega}u-N|\leq \ep$ in $B(i,R)$. By the definition of $\delta_1$, from Proposition \ref{2.2} we have $\nabla_1 u|_{B(i_0,R_1)}\geq M-\ep$, which implies $\nabla_{1}u|_{B(i,R)}\geq M-\ep$ since $B(i,R)\subseteq B(i_0,R_1)$. So $|\nabla_{1}u-M|\leq \ep$ in $B(i,R)$. We finish the proof.
\qed


\emph{Proof of Theorem \ref{main}.} If $u$ is not constant, we may assume $M=\sup \nabla_1 u>0$. Then applying Lemma \ref{overlap2} for this $M$, we obtain $\ep$ and $R$ depending on $M$. For $M,\ep,R$, we apply Lemma \ref{MN} to obtain $i,N$. Since $i,N$ are arbitrary in Lemma \ref{overlap2}, we see that $B(i,R)$ has an overlap of positive area, which contradicts to the completeness of the PL metric.
\qed

\subsection{Proof of Lemma \ref{mp}}\label{p1}

In this section, we prove Lemma \ref{mp}. As in Section \ref{sec-outline-proof}, let $H=H(i_0;i_1,i_2,i_3,i_4,i_5,i_6)$ be a hexagon of center $i_0$ with PL metric $l_0\equiv 1$, see Figure \ref{fig-8}. Let $l$ be a generalized metric, $u$ be the conformal factor. Denote $l=u*l_0$. Suppose the curvature $K_u$ at $i_0$ is zero.

First, we show a lemma to avoid the degeneracy of the triangles. Roughly speaking, the degeneracy of a smaller triangle can be controlled by the degeneracy of a bigger triangle.
\begin{lemma}\label{tri}
Suppose $\{\bar{l}_1,\bar{l}_2,\bar{l}_1\bar{l}_2\}$ forms a generalized triangle. Suppose $l_1,l_2>0$ and $l_1\leq \bar{l}_1,~l_2\leq \bar{l}_2.$ Then

(1) Either $l_1+l_2> l_1l_2$, or $l_1=\bar{l}_1,~l_2=\bar{l}_2.$

(2) if in addition $l_2=\bar{l}_2$, then either $l_2+l_1l_2> l_1$ or $l_1=\bar{l}_1$.
\end{lemma}
\begin{proof} To show (1), if $l_1+l_2\leq l_1l_2$, we have $l_1,l_2> 1$ and
\begin{eqnarray*}
0\geq \bar{l}_{1}\bar{l}_{2}-\bar{l}_{1}-\bar{l}_{2}=(\bar{l}_{1}-1)(\bar{l}_{2}-1)-1
\geq(l_1-1)(l_2-1)-1\geq 0.
\end{eqnarray*}
So $l_1=\bar{l}_1,~l_2=\bar{l}_2.$

To show (2), if $l_2+l_1l_2\leq l_1$, we have $\bar{l}_{2}=l_2< 1$ and
$$\bar{l}_{1}\geq l_1\geq\frac{l_2}{1-l_2}=\frac{\bar{l}_{2}}{1-\bar{l}_{2}}\geq \bar{l}_1.$$
So $l_1=\bar{l}_1.$
\end{proof}


The following lemma is the calculation of the derivative of the angle function with respect to the PL conformal factor $u$.
\begin{lemma}\label{basic}
Let $\triangle ijk$ be a triangle with PL metric $u*l_0$, where $l_0\equiv 1$ and the conformal factor $u=(u_i,u_j,u_k)$. Let $\theta_i,\theta_j,\theta_k$ be the angle at the vertex $i,j,k$ respectively. Then
$$\frac{\p \theta_i}{\p u_j}=\cot \theta_k,\quad\frac{\p \theta_i}{\p u_k}=\cot \theta_j,\quad\frac{\p \theta_i}{\p u_i}=-\cot \theta_j-\cot \theta_k.$$
\end{lemma}
\pf It is from direct calculation or see \cite{Luo}.\qed

As a corollary, we obtain the derivative of the curvature function.
\begin{corollary}\label{derivative}
For $u\in \mc{T}$, $\frac{\partial K_u}{\partial u_j}=-(\cot \theta_{\angle i_0i_{j+1}i_j}+\cot \theta_{\angle i_0i_{j-1}i_j})$. In particular, for $u\in \mc{D}$, $\frac{\partial K_u}{\partial u_j}\leq 0$.
\end{corollary}

Now we prove the maximum principle Lemma \ref{mp} for general $n$. The notations $\mc{T}$ and $\mc{D}$ are similarly defined as in Section \ref{sec-outline-proof} just replacing $6$ with $n$. Let $\bar{u},\underline{u}\in \mc{D}$. Suppose $\bar{u}\geq \underline{u}$. Denote
$$\mc{T}_{\bar{u},\underline{u}}=\{u\in \mc{T}: \underline{u}\leq u\leq \bar{u} \},\qquad \mc{D}_{\bar{u},\underline{u}}=\{u\in \mc{D}: \underline{u}\leq u\leq \bar{u} \}.$$
Then $\mc{D}_{\bar{u},\underline{u}}$ is clearly bounded and closed.
\begin{lemma}\label{key}(Lemma \ref{mp} for general $n$)
Let $\bar{u},\underline{u}\in \mc{D}$ with $\bar{u}_0=\underline{u}_0=0$. If $K_{\underline{u}}=K_{\bar{u}}=0$ and $\bar{u}\geq \underline{u}$, then $\bar{u}=\underline{u}$.
\end{lemma}
\begin{proof}

Claim 0: Let $u\in\mc{D}_{\bar{u},\underline{u}}$, $u\neq \bar{u}$, $K_u\leq 0$. Then there exists $v\in\mc{D}_{\bar{u},\underline{u}}$, such that $v\geq u$ and $K_{v}<K_{u}$.

We show Lemma \ref{key} under Claim 0. We first apply Claim 0 to $u=\underline{u}$. Then there exists $\tilde{v}\in\mc{D}_{\bar{u},\underline{u}}$, such that $\ti{v}\geq \underline{u}$ and $K_{\ti{v}}<K_{\underline{u}}=0$. Denote
$$S=\{u\in\mc{D}_{\bar{u},\underline{u}}:~K_u\leq K_{\ti{v}}\}.$$
For $u\in S$, denote
$$||\bar{u}-u||_1=\sum\limits_{i=1}^{n}|\bar{u}_i-u_i|.$$
Set $A=\inf\limits_{u\in S}||\bar{u}-u||_1.$ We suppose $u^i\in S$, $||\bar{u}-u^i||_1\rightarrow A$ and $u^i\rightarrow u^{\infty}.$ Since $\mc{D}_{\bar{u},\underline{u}}$ is closed and $K_u$ is continuous in $u$, $S$ is also closed. So $u^{\infty}\in S.$ If $A=0$, then $u^{\infty}=\bar{u}$. Then
$$0>K_{\ti{v}}\geq K_{u^{\infty}}=K_{\bar{u}}=0.$$
Contradiction. If $A>0$, we apply Claim 0 to $u=u^{\infty}$. Then there exists $v\in\mc{D}_{\bar{u},\underline{u}}$, such that $v\geq u^{\infty}$ and $K_{v}<K_{u^{\infty}}$. So $v\in S$ and $||\bar{u}-v||_1<||\bar{u}-u^{\infty}||_1=A$. Contradiction. Now we prove Claim 0.

Claim 1: For $u\in\mc{D}_{\bar{u},\underline{u}}$ $u\not=\bar{u},K_u\le 0$, if $u_j<\bar{u}_j$ and $\alpha_{i_0i_j}<\pi$ for some $j$, then Claim 0 holds.

we consider
$$v=(u_1,\cdots,u_j+\ep,\cdots,u_n), ~\ep \text{ small enough}.$$
If the triangles $\triangle i_0i_ji_{j+1}$ and $\triangle i_0i_ji_{j-1}$ are both non-degenerate with respect to $v$, then
from Lemma \ref{basic}, we have
\begin{eqnarray*}
\frac{\partial \theta_{\angle i_0i_ji_{j+1}}}{\partial u_j}&=&-(\cot \theta_{\angle i_ji_0i_{j+1}}+\cot\theta_{\angle i_ji_{j+1}i_0})< 0,\\
\frac{\partial \theta_{\angle i_0i_ji_{j-1}}}{\partial u_j}&=&-(\cot \theta_{\angle i_ji_0i_{j-1}}+\cot\theta_{\angle i_ji_{j-1}i_0})< 0.
\end{eqnarray*}
Since $\alpha_{i_0i_j}<\pi$, 
we see for $\ep$ small enough, $v\in\mc{D}_{\bar{u},\underline{u}}$. From Lemma \ref{basic} and $\alpha_{i_0i_j}<\pi$,
$$\frac{\partial K_u}{\partial u_j}=-(\cot \theta_{\angle i_0i_{j+1}i_j}+\cot \theta_{\angle i_0i_{j-1}i_j})<0.$$ Let $\ep$ small enough, we obtain the desired result.

Suppose one of the triangle, say $\triangle i_0i_ji_{j+1}$ degenerates with respect to $u$, while the other does not. Recall $u_0=\bar{u}_0=0$, so $l_{i_0i_j}=e^{u_j}$, $l_{i_0i_{j}}l_{i_0i_{j+1}}=l_{i_ji_{j+1}}$. From the part (1) of Lemma \ref{tri}, noticing $\bar{u}\geq u$ and $\bar{u}_j>u_j$, we rule out $l_{i_0i_j}l_{i_0i_{j+1}}=l_{i_0i_j}+l_{i_0i_{j+1}}$. From $\alpha_{i_0i_j}<\pi$, we rule out $l_{i_0i_{j+1}}+l_{i_0i_j}l_{i_0i_{j+1}}=l_{i_0i_{j}}$. So the only possibility is $l_{i_0i_j}+l_{i_0i_j}l_{i_0i_{j+1}}=l_{i_0i_{j+1}}$. In this case, increase $l_{i_0i_j}$ slightly, then both triangle would become non-degenerate and $K_u$ would become negative. Suppose both $\triangle i_0i_ji_{j+1}$ and $\triangle i_0i_ji_{j-1}$  degenerate, obviously, increasing $l_{i_0i_j}$ works for both.  Then as in the non-degenerate case, we finish the proof of Claim 1.

Claim 2: There exists $j\in\{1,\cdots,n\}$ such that $u_j<\bar{u}_j$ and $\alpha_{i_0i_j}<\pi$.

Now we prove Claim 2 to finish the whole proof. If Claim 2 fails, we may suppose that $u_j<\bar{u}_j$ implies $\alpha_{i_0i_j}\geq\pi$ for every $j$. We show the configuration is impossible. Let $S=\{u_i: ~u_i<\bar{u}_i\}$, $k=\#S$. Then together with $0\geq K_u=2\pi-\theta_u,$
$$n\pi=\text{the sum of all angles of the n triangles}\geq k\cdot \pi+\theta_{u}\geq (k+2)\pi.$$
So $k\leq n-2$. Then there are at least two edges satisfying $u_j=\bar{u}_j$. Consider the configuration $(i_0;i_{s},\cdots,i_{s+l})$, i.e. the triangles $\triangle i_{0}i_{s}i_{s+1}$, $\cdots$, $\triangle i_{0}i_{s+l-1}i_{s+l}$, satisfying $u_{s}=\bar{u}_{s},u_{s+l}=\bar{u}_{s+l}$, for some $l\geq 1$ and $u_{s+a}<\bar{u}_{s+a},a=1,\cdots,l-1$. 
We show the configuration is impossible. 
For simplicity we assume
\begin{equation*}
u_1=\bar{u}_1,~u_2<\bar{u}_2,\cdots,~u_{m-1}<\bar{u}_{m-1},~u_{m}=\bar{u}_m,\qquad \alpha_{i_0i_{2}},\cdots,\alpha_{i_0i_{m-1}}\geq\pi.
\end{equation*}

The idea is to construct a flow $u^t$ in $\mc{T}_{\bar{u},\underline{u}}$ (in fact the restriction on $(i_0;i_{1},\cdots,i_{m})$), satisfying the following conditions (we omit the superscript $t$ if there is no confusion):\\
(1) $u_2,\cdots,u_{m-1}$ are increasing and $u_1,u_{m}$ are fixed;\\
(2) $\alpha_{i_0i_2},\cdots,\alpha_{i_0i_{m-2}}$ are invariant and $\alpha_{i_0i_{m-1}}$ is increasing.\\
This flow will eventually touch the boundary of $\mc{T}_{\bar{u},\underline{u}}$. We can rule out the possibility that some triangles degenerate. So it must happen $u_j=\bar{u}_j$ for some $j$ when the flow stops. Then repeat this procedure, we obtain $u_j=\bar{u}_j$ for $j=1,\cdots,n$. Then it is a contradiction since $\bar{\alpha}_{i_0i_{m-1}}\leq\pi$ while $\alpha_{i_0i_{m-1}}>\pi$.


Suppose the configuration happens. For non-degenerate triangles, denote
\begin{eqnarray*}
A_j&=&\cot \theta_{\angle i_ji_0i_{j-1}}+\cot \theta_{\angle i_ji_0i_{j+1}},\\
B_j&=&\cot \theta_{\angle i_ji_0i_{j-1}}+\cot \theta_{\angle i_ji_{j-1}i_{0}},\\
C_j&=&\cot \theta_{\angle i_ji_0i_{j+1}}+\cot \theta_{\angle i_ji_{j+1}i_{0}}.
\end{eqnarray*}
Notice that $B_j,C_j>0$. And if $\alpha_{i_0i_j}\geq \pi$, then $A_j\geq B_j+C_j$ .
We consider the following flow $u^t=(u_1^t,\cdots,u_m^t)$.
\begin{eqnarray*}
\frac{d}{dt}u^t_i&=&X_i,\quad i=1,\cdots, m,\\
u_i^0&=&u_i,\quad i=1,\cdots, m,
\end{eqnarray*}
where $X_1,\cdots,X_{m}$ is the solution to the following equation system
\begin{eqnarray*}
X_1&=&0,\\
X_2&=&1,\\
A_2X_2-B_3X_3&=&0,\\
A_{s}X_{s}-B_{s+1}X_{s+1}-C_{s-1}X_{s-1}&=&0,~s=3,\cdots, m-2,\\
X_{m}&=&0.
\end{eqnarray*}
Notice that $A_s,B_s,C_s,X_s$ also depend on $u^t$.

Suppose triangles $\triangle i_0i_jj_{j+1}$, $j=1,\cdots,m-1$ are not degenerate in the flow. Then from Lemma \ref{basic},
$$\frac{d\alpha_{i_0i_j}}{dt}=A_j\frac{du_j}{dt}-B_{j+1}\frac{du_{j+1}}{dt}-C_{j-1}\frac{du_{j-1}}{dt}.$$
From the construction of $X_j$, $\frac{d}{dt}\alpha_{i_0i_j}=0$, hence $\alpha_{i_0i_j}^t=\alpha_{i_0i_j}^0\geq \pi$ for $j=2,\cdots,m-2$. Then $A_j\geq B_j+C_j>0$, $j=2,\cdots,m-2$.

Fix $j\in\{3,\cdots,m-1\}$, we claim: Let $f_0=1$, $f_1=A_j$,
$f_{s+1}=A_{j-s}f_s-B_{j-s+1}C_{j-s}f_{s-1}$, $s=2,\cdots,j-2.$
Suppose $A_j>0$ and $A_{j-s}\geq B_{j-s}+C_{j-s}$ for $s=1,\cdots,j-2$. Then $f_s>0$, $f_s>B_{j-s}f_{s-1}$ for $s=1,\cdots,j-1$.

We prove the claim by induction. For $s=1$, the claim holds. Suppose the claim holds for $s$. We consider $s+1$. Since $f_{s}>0$,
$$f_{s+1}=A_{j-s}f_s-B_{j-s+1}C_{j-s}f_{s-1}\geq (B_{j-s}+C_{j-s})f_s-B_{j-s+1}C_{j-s}f_{s-1}.$$
So $f_{s+1}-B_{j-s}f_{s}\geq (f_{s}-B_{j-s+1}f_{s-1})C_{j-s}>0$ by the assumption. Then $f_{s+1}>B_{j-s}f_s>0$. We finish the proof of the claim.

Now we show $B_{j+1}X_{j+1}=A_jX_j-C_{j-1}X_{j-1}>0$ for $j=2,\cdots,m-2$. We observe that $f_{s+1}X_{j-s}-C_{j-s-1}f_sX_{j-s-1}>0$ is equivalent to $f_{s+2}X_{j-s-1}-C_{j-s-2}f_{s+1}X_{j-s-2}>0$. In fact,
\begin{\eq}
f_{s+1}X_{j-s}-C_{j-s-1}f_sX_{j-s-1}&=&f_{s+1}B_{j-s}^{-1}(A_{j-s-1}X_{j-s-1}-C_{j-s-2}X_{j-s-2})-C_{j-s-1}f_sX_{j-s-1}\\
&=& B_{j-s}^{-1}\big((A_{j-s-1}f_{s+1}-B_{j-s}C_{j-s-1}f_s)X_{j-s-1}-C_{j-s-2}f_{s+1}X_{j-s-2}\big)\\
&=& B_{j-s}^{-1}(f_{s+2}X_{j-s-1}-C_{j-s-2}f_{s+1}X_{j-s-2}).
\end{\eq}
So by induction, to show $A_jX_j-C_{j-1}X_{j-1}>0$, by letting $s=j-3$, we only need to show $f_{j-2}X_{3}-C_2f_{j-3}X_2>0$. Since $A_2X_2-B_3X_3=0$ and $X_2=1$, it is equivalent to $f_{j-1}=A_2f_{j-2}-B_3C_2f_{j-3}>0$. Since we have obtained $A_{s}\geq B_s+C_s>0$ for $s=2,\cdots,m-2$, then by the claim above, we obtain $f_{j-1}>0$ for $j=2,\cdots,m-2$. So $X_j>0$ for $j=1,\cdots,m-1$, which means $u_j$ is increasing for $j=1,\cdots,m-1$.

Next we show $\alpha_{i_0i_{m-1}}$ is increasing, from Lemma \ref{basic}, we have
$$\frac{d\alpha_{i_0i_{m-1}}}{dt}=A_{m-1}\frac{du_{m-1}}{dt}-C_{m-2}\frac{du_{m-2}}{dt}=A_{m-1}X_{m-1}-C_{m-2}X_{m-2}.$$
As the discussion above, if $A_{m-1}>0$ then $\frac{d\alpha_{i_0i_{m-1}}}{dt}>0$.
So if $\alpha_{i_0i_{m-1}}\geq \pi$, then $\frac{d\alpha_{i_0i_{m-1}}}{dt}>0$. Since $\alpha^0_{i_0i_{m-1}}\geq \pi$, we have $\alpha^t_{i_0i_{m-1}}> \pi$ and increasing for $t>0$.

Let $T$ be the maximal existence time of $u^t$ in $\mc{T}_{\bar{u},\underline{u}}$. Since $\frac{du_2}{dt}=1$, $T$ must be finite. Notice that it may happen $T=0$. Since $u^t$ is increasing and bounded, when $t$ approaches to $T$, it has a limit $\hat{u}$ touching the boundary of $\mc{T}_{\bar{u},\underline{u}}$.
There are two situations. The first one is $\hat{u}_j<\overline{u}_j$, $j=1,\cdots,m-1$, and then one of the triangles $\triangle i_0i_ji_{j+1}$, $j=1,\cdots,m-1$ degenerates. The second one is $\hat{u}_{j_0}=\overline{u}_{j_0}$ for some $j_0\in\{2,\cdots,m-1\}$.

We rule out the first situation. Suppose $\triangle i_0i_{j_0}i_{j_0+1}$ degenerates. Suppose $\hat{l}_{i_0i_{j_0+1}}\leq \hat{l}_{i_0i_{j_0}}$. Then from the part (1) of Lemma \ref{tri}, the only possibility is $\hat{l}_{i_0i_{j_0+1}}+\hat{l}_{i_0i_{j_0}}\hat{l}_{i_0i_{j_0+1}}=\hat{l}_{i_0i_{j_0}}$. So $\hat{\theta}_{\angle i_0i_{j_{0}}i_{j_0+1}}=0$. Since $\hat{\alpha}_{i_0i_{j_0+1}}\geq \pi$, we have $\hat{\theta}_{\angle i_0i_{j_{0}+2}i_{j_0+1}}=\pi$. Then we can repeat this procedure until $\triangle i_{m-1}i_0i_{m}$. Notice $\hat{u}_m=\overline{u}_m$ and $\hat{u}_{m-1}<\overline{u}_{m-1}$, it contradicts to the part (2) of Lemma \ref{tri}.

So it must be the second situation. Suppose $\hat{u}_{j_0}=\overline{u}_{j_0}$. Then we repeat this procedure to all the configurations. Then finally we obtain $m=2$ and $\tilde{u}_{j}=\bar{u}_j$ for every $j=1,\cdots,n$, where $\tilde{u}$ is the conformal factor at the time when the procedure stops. If throughout this procedure, the flow never runs, then $u=u^0=\tilde{u}=\bar{u}$, contradiction. If the flow runs for a while, for example the case we discuss above, then $\hat{\alpha}_{i_0i_{m-1}}>\pi$. Notice that this condition preserves during the procedure, so $\tilde{\alpha}_{i_0i_{m-1}}>\pi$. Therefore it contradicts to $\tilde{u}=\bar{u}$ but $\bar{\alpha}_{i_0i_{m-1}}\leq\pi$.
We finish the proof.
\end{proof}

\subsection{Proof of Proposition \ref{2.2}}\label{p2}
In this section, we prove Proposition \ref{2.2}. Let $(T,l_0)$ be the standard hexagonal triangulated plane with $l_0\equiv 1$. Let $(T,l)$ be a triangulated plane conformal to the standard one. Let $l=u*l_0$.

In \cite{WGS}, Wu, Gu and Sun introduced the notion of quasi-harmonicity.
\begin{definition}\label{quasih}
Let $f$ be a function on $V$. For $m>0$, we call $f$ is quasi-harmonic with harmonic factor $m$ at $i\in V$, if there is a weighted average $f(i)=\sum\limits_{j\sim i}m_{j}f(j)$, where $m_j$ depends on $i$ and $\sum\limits_{j\sim i}m_{j}=1$, such that $m_j\geq m$ for $j\sim i$. We call $f$ is quasi-harmonic with harmonic factor $m$ if $f$ is quasi-harmonic with harmonic factor $m$ for all $i\in V$.
\end{definition}
The quasi-harmonicity of $f$ means $f(i)$ is a weighted average of the values of its neighbors, but not being too close to the maximum or minimum. In particular, if $f(i)$ is close to the supremum of $f$, then its neighbors must also be close to the supremum. The next technical lemma gives a sufficient condition to show the quasi-harmonicity.
\begin{lemma}\label{aver}
Given $\ep,M>0$. Let $a_i\in \mathbb{R}$, $|a_i|\leq M$, $i=1,\cdots, 6$. Suppose $\max\limits_i a_i\geq \ep$ and $\min\limits_{i} a_i\leq -\ep$. Then there exists $m>0$ depending on $\ep$ and $M$, such that for some $m_i\geq m$, $i=1,\cdots, 6$ with
$\sum\limits_{i=1}^{6}m_i=1$ we have, $$\sum\limits_{i=1}^{6}m_ia_i=0.$$
\end{lemma}
\begin{proof}
Suppose $a_1=\min\limits_i a_i,~a_6=\max\limits_{i} a_i$. Set $\bar{a}=\frac{1}{4}\sum\limits_{i=2}^5a_i$. We assume $\bar{a}\geq 0$. Set $\tilde{a}=\frac{1}{2}(\bar{a}+a_6)\geq \frac{\ep}{2}$. Then $$0=\ti{a}a_1+(-a_1)\ti{a}=\ti{a}a_1+\frac{-a_1}{8}\sum\limits_{i=2}^5a_i+\frac{-a_1}{2}a_6.$$
Let $m_1=\frac{\ti{a}}{\ti{a}-a_1},~m_2=\cdots=m_6=\frac{-a_1}{\ti{a}-a_1}$. Then $\frac{\ti{a}}{\ti{a}-a_1}\geq \frac{\frac{\ep}{2}}{2M}$ and $\frac{-a_1}{\ti{a}-a_1}\geq \frac{\ep}{2M}$.
We finish the proof.
\end{proof}

Let $H=(i_0;i_1,i_2,i_3,i_4,i_5,i_6)$ be a hexagon centered at $i_0$. Let $l$ and $\ti{l}$ be two Delaunay metrics on $H$ conformal to $l_0$ with zero curvature at $i_0$. Let $l=u*l_0,\ti{l}=\ti{u}*l_0$. Set $u_j=u_{i_j},\ti{u}_j=\ti{u}_{i_j}$, $j=0,1,\cdots,6$ for short. Define $\delta u_{j}=\ti{u}_{j}-u_{j}.$ 
The next lemma shows that if $\delta u$ is bounded, then either $\delta u$ is quasi-harmonic at $i_0$ or $\delta u_0$ is close to the values of the neighbors of $i_0$.
\begin{lemma}\label{hex1}
Suppose $|\delta u_j|\leq M$, $j=0,\cdots,6$. Then for any $\epsilon>0$, there exists a constant $m>0$ depending on $\ep$ and $M$, such that either

(1) $|\delta u_j-\delta u_0|\leq \ep$ for $j=1,\cdots,6$; or

(2) $\delta u$ is quasi-harmonic at $i_0$ with harmonic factor $m$, i.e. there exists $m_j\geq m$, $j=1,\cdots,6$, $\sum\limits_{j=1}^6m_j=1$ such that $\sum\limits_{j=1}^6 m_j(\delta u_j-\delta u_0)=0.$
\end{lemma}
\begin{proof}
Suppose situation (1) fails, we show situation (2) holds. Notice that $\delta u_j-\delta u_0$ is invariant under the transformation, $(u_j,\ti{u}_j)\mapsto (u_j+c,\ti{u}_j+\ti{c})$ for $j=0,\cdots,6$,
we may assume $u_0=\ti{u}_0=0$. Then after the similar transformation, $|\delta u_j|\leq 2M$.
By the assumption, we assume $\max\limits_{j}\delta u_j\geq \ep$. 
Denote $\mc{D}^0=\{u\in\mc{D}~|~K_{u}=0,u_0=0\}$.
From Lemma \ref{ratio}, we see $\mc{D}^0$ is compact. We claim there is a constant $\ep_1>0$ depending on $\ep$, such that $\min\limits_{j}\delta u_j\leq -\ep_1$. If it is false, there exists a sequence of pairs $(u^n,\ti{u}^n)\in \mc{D}^0\times\mc{D}^0$
such that
$$\max\limits_{j}\delta u_j\geq \ep\text{ and } \min\limits_{j}\delta u_j\geq -\frac{1}{n}.$$
Since $\mc{D}^0$ is compact, by taking the subsequence, we may assume $u^n\rightarrow u^{\infty}$ and similarly $\ti{u}^n\rightarrow \ti{u}^{\infty}$. Then we have $$u^{\infty},\ti{u}^{\infty}\in \mc{D},\quad u^{\infty}\leq \ti{u}^{\infty},\quad u^{\infty}\neq \ti{u}^{\infty}~\text{ and }~K_{u^{\infty}}=K_{\ti{u}^{\infty}}=0,$$
which contradicts to Lemma \ref{key}. So we have $\max\limits_{j}\delta u_j\geq \ep_1$ and $\min\limits_{j}\delta u_j\leq -\ep_2$. Then Lemma \ref{hex1} follows from Lemma \ref{aver}.
\end{proof}

Recall the vertices of $T$ are spanned by $1$ and $\omega=\frac{1+\sqrt{-3}}{2}$. For either $c=1$ or $\omega$, define the gradient $\nabla_c u(i)=u(i+c)-u(i)$.
From Lemma \ref{ratio}, $|\nabla_c u|$ has a universal bound. For a hexagon $H=(i_0;i_1,i_2,i_3,i_4,i_5,i_6)$ in $T$, applying Lemma \ref{hex1} to the case $\ti{u}(i)=u(i+c)$, i.e. $\delta u=\nabla_c u$, we obtain
\begin{corollary}\label{hex2}
For any $\epsilon>0$, there exists a constant $m=m(\epsilon)>0$ depending on $\ep$, such that either

(1) $|\nabla_c u(i_j)-\nabla_c u(i_0)|\leq \ep$ for $j=1,\cdots,6$; or

(2) $\nabla_c u$ is quasi-harmonic at $i_0$ with harmonic factor $m$.
\end{corollary}

Now we are ready to prove Proposition \ref{2.2}. Let $c$ be $1$ or $\omega$. Roughly speaking, from the comments below Definition \ref{quasih}, if $\nabla_c u(i)$ is close to the supremum, then either situation (1) or (2) of Corollary \ref{hex2} implies the values of its neighbors are also close to the supremum. Recall $B(i,R)$ denotes the set of triangles whose vertices can be connected with $i$ by a path of at most $R$ edges.
\begin{proposition}\label{grad}
For any $\ep>0$, $R\in \mathbb{N}$, there exists a constant $\delta>0$ depending on $\ep,R$, such that for any $M>0$, $i\in V$, if
$$\nabla_c u|_{B(i,R)}\leq M \text{ and } \nabla_c u(i)\geq M-\delta,$$
then $\nabla_c u|_{B(i,R)}\geq M-\ep$.
\end{proposition}
\begin{proof}
For any $\epsilon>0$, let $m(\epsilon)$ be the constant given
by Corollary
\ref{hex2}. Let $\epsilon_{R}=\epsilon$, $\epsilon_{j-1}=\min\{\frac{\epsilon_{j}}{2},\epsilon_{j}m(\frac{\epsilon_{j}}{2})\}$
for any $j=1,\cdots,R$. We let $\delta=\epsilon_{0}$. Then we
prove by induction that $\nabla_{c}u|_{B(i,j)}\ge M-\epsilon_{j}$.
\begin{itemize}
\item $j=0$, $\nabla_{c}u(i)\ge M-\delta=M-\epsilon_{0}$;
\item Suppose the conculsion holds for $j=k$, i.e.
\[
\nabla_{c}u|_{B(i,k)}\ge M-\epsilon_{k}=M-\min\{\frac{\epsilon_{k+1}}{2},\epsilon_{k+1}m(\frac{\epsilon_{k+1}}{2})\}.
\]

Then when $j=k+1$, either for any $p\in\partial B(i,k+1)$, which
is a neighbor of $q\in B(i,k)$, there holds
\[
\nabla_{c}u(p)\ge\nabla_{c}u(q)-\frac{\epsilon_{k+1}}{2}\ge M-\epsilon_{k}-\frac{\epsilon_{k+1}}{2}\ge M-\epsilon_{k+1},
\]
or $\nabla_{c}u$ is quasi-harmonic at $q$ with factor $m(\frac{\epsilon_{k+1}}{2})$,
which again implies
\[
\nabla_{c}u(p)\ge M-\frac{\epsilon_{k}}{m(\frac{\epsilon_{k+1}}{2})}\ge M-\epsilon_{k+1}.
\]
\item By induction we know that
\[
\nabla_{c}u|_{B(i,j)}\ge M-\epsilon_{j},
\]
and in particular $\nabla_{c}u|_{B(i,R)}\ge M-\epsilon_{R}.$
\end{itemize}
\end{proof}

\section{Viewpoints from hyperbolic geometry}\label{hyp}
\subsection{PL conformal vs. hyperbolic geometry}
We first recall some basic facts in hyperbolic geometry. Identifying the unit disk $\mathbb{D}\subset\mathbb{C}$ with the set $\{(x_1,x_2,x_3)\in\mathbb{R}^3: x_1^2+x_2^2<1,~x_3=0\}$
in $\mathbb{R}^3$, and mapping $\mathbb{D}$ under the stereographic projection $\Pi$ with respect to the south pole $(0,0,-1)$, we obtain the upper half of the unit sphere
$$\mathbb{S}^2_{+}=\{x=(x_1,x_2,x_3)\in\mathbb{R}^3: |x|=1, x_3>0\}.$$

Thus, composing $\Pi$ with the projection $P: (x_1, x_2, x_3)\mapsto(x_1, x_2, 0)$, we obtain a homeomorphism $P\Pi$ from $\mathbb{D}$ onto itself. We can extend $P\Pi$ continuously to the boundary $S^1$ of $\mathbb{D}$ by setting $P\Pi(x)=x$. The geodesic lines on the Poincar\'e unit disk model $(\mathbb{D}, \frac{2|dz|}{1-|z|^2})$ of $\mathbb{H}^2$ are mapped under $P\Pi$ to Euclidean segments with the same end points. The disc $\mathbb{D}$ with the metric induced by $P\Pi$ from $(\mathbb{D}, \frac{2|dz|}{1-|z|^2})$ is called the \emph{Klein model} (also called the projective model) of the hyperbolic plane $\mathbb{H}^2$. See Figure \ref{fig-g2}. Similarly, we obtain the Klein model on the interior of any circle in $\mathbb{C}$.

\begin{figure}[htbp]\centering
\includegraphics[width=0.7\textwidth]{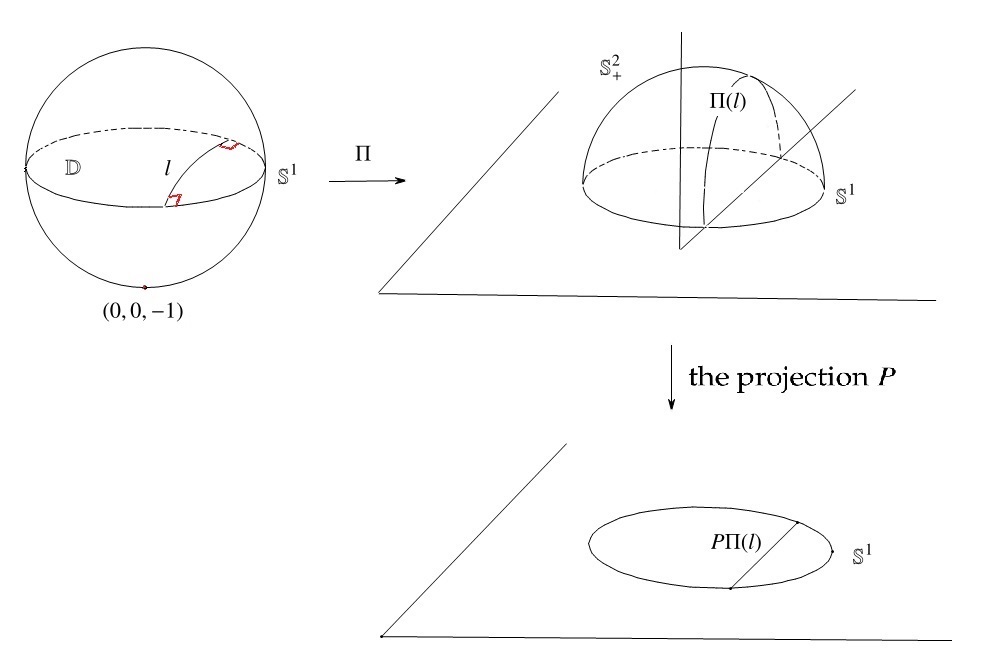}
\caption{The stereographic projection and the Klein model}\label{fig-g2}
\end{figure}

On any PL surface $(\Sigma, T, l)$, Bobenko, Pinkall and Springborn \cite{BPS} constructed a complete natural hyperbolic metric with cusps. Consider a Euclidean triangle with its circumcircle. Interpret the interior of the circumcircle as the Klein model, then the Euclidean triangle becomes an ideal hyperbolic triangle, that is, a hyperbolic triangle with vertices at infinity. This construction equips any Euclidean triangle (minus its vertices) with a hyperbolic metric. If it is performed on all triangles in the triangulation $T$, then the hyperbolic metrics induced on the individual triangles fit together, so $\Sigma\backslash V$ is equipped with a hyperbolic metric with cusps at the vertices. Thus, $T$ becomes an ideal triangulation of a hyperbolic surface $\Sigma$ with cusps $V$.

\begin{figure}[htbp]\centering
\includegraphics[width=0.4\textwidth]{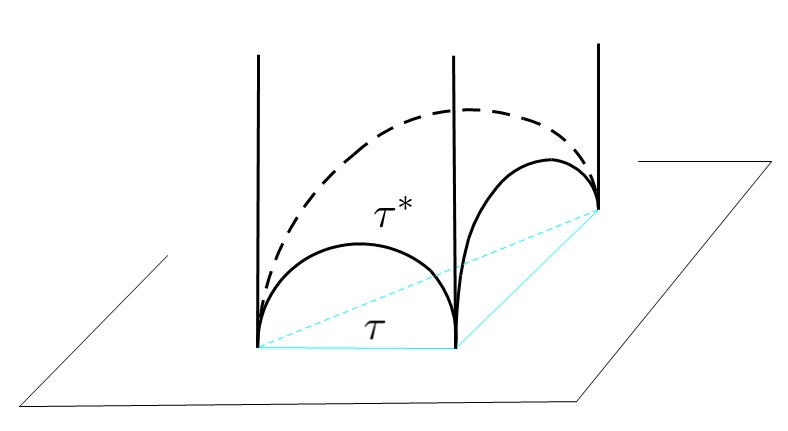}
\caption{A Euclidean triangle and the corresponding ideal one}\label{fig-g3}
\end{figure}

Gu, Luo, Sun and Wu \cite{GLSW} further expressed the above construction more geometrically. Consider $\mathbb{C}$ as the sphere at the infinity of the hyperbolic $3$-space $\mathbb{H}^3=\mathbb{C}\times\mathbb{R}_{>0}$. For each Euclidean triangle $\tau$ (considered as a subset of $\mathbb{C}$), let $\tau^*$ be the ideal hyperbolic triangle in $\mathbb{H}^3$ having the same set of vertices as that of $\tau$. Geometrically, $\tau^*$ is exactly the convex hull in $\mathbb{H}^3$ spanned by the three vertices of $\tau$. See Figure \ref{fig-g3}. If $\tau_1$, $\tau_2$ are two Euclidean triangles in $T$ glued along their common edge by a Euclidean isometry $f$, then one glues $\tau_1^*$ and $\tau_2^*$ along their corresponding edges by $\tilde{f}$ (the Poincar\'e extension of $f$). See Figure \ref{fig-g4}. In this way, one produces a hyperbolic metric $l^*$ on $\Sigma\setminus V$ with cusps $V$.
For any (oriented) edge $ij$, let $ijk$, $ijl$ be the two Euclidean triangles in $T$ that adjacent to $ij$. Penner \cite{Penner} showed that Thurston's shear coordinate at $ij$ of the hyperbolic metric $l^*$ constructed above is $\ln (l_{jl}l_{ik}/l_{il}l_{jk})$. At each vertex $i$, it is easy to see that all shear coordinates $\ln (l_{jl}l_{ik}/l_{il}l_{jk})$ sum to zero. By Thurston \cite{Thurston} $\S3.7$-$\S3.9$, the hyperbolic metric $l^*$ constructed above is complete.

\begin{figure}[htbp]\centering
\includegraphics[width=0.75\textwidth]{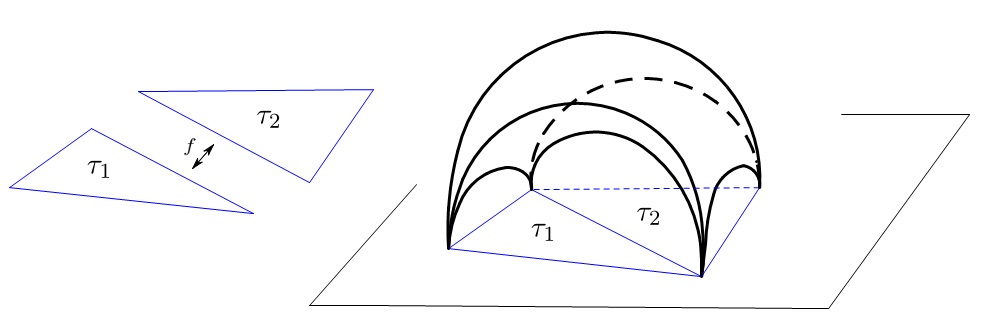}
\caption{Gluing Euclidean triangles and the corresponding ideal ones}\label{fig-g4}
\end{figure}

\begin{theorem}[Bobenko-Pinkall-Springborn]
\label{thm-BPS}
Two PL metrics $l$ and $\tilde{l}$ are PL conformal if and only if the corresponding complete hyperbolic metrics with cusps $l^*$ and $\tilde{l}^*$ are isometric.
\end{theorem}
\begin{proof}
We give a proof here by following Bobenko, Pinkall, Springborn \cite{BPS} and Gu, Luo, Sun, Wu \cite{GLSW}\cite{Luo-book}. By Thurston's theory of hyperbolic surfaces, the hyperbolic metrics $l^*$ and $\tilde{l}^*$ are isometric if and only if their shear coordinates are the same at each edge $e\in T$. For any (oriented) edge $ij$, let $ijk$, $ijl$ be the two Euclidean triangles in $T$ that adjacent to $ij$. The shear coordinate at the edge $ij$ is $\ln \text{lcr}_{ij}$, where
$$\text{lcr}_{ij}=\frac{l_{il}l_{jk}}{l_{lj}l_{ki}}$$
is the length-cross-ratio at $ij$ (we refer $\S2.3$ \cite{BPS} for more about lcr). If $l$ and $\tilde{l}$ are PL conformal, that is, $\tilde{l}=u*l$ for some $u:V\rightarrow\mathbb{R}$, then obviously $\widetilde{\text{lcr}}_{ij}=\text{lcr}_{ij}$ for each $ij$. It follows that $l^*$ and $\tilde{l}^*$ are isometric. Conversely, if $l^*$ and $\tilde{l}^*$ are isometric, then $\widetilde{\text{lcr}}=\text{lcr}$. For each triangle $ijk$, one may find a unique solution $u_i$, $u_j$, $u_k$ so as $\tilde{l}_{st}=e^{u_s+u_t}l_{st}$, $st\in\{ij, jk, ki\}$. For another triangle $ijl$ which sharing a common edge with $ijk$, one may also find a unique solution $u_i'$, $u_j'$, $u_l'$ so as $\tilde{l}_{st}=e^{u_s'+u_t'}l_{st}$, $st\in\{ij, jl, li\}$. From $\widetilde{\text{lcr}}_{ij}=\text{lcr}_{ij}$, one easily see $u_i'=u_i'$. This implies there is a global defined function $u:V\rightarrow\mathbb{R}$ so that $\tilde{l}=u*l$ and hence $l$ and $\tilde{l}$ are PL conformal.
\end{proof}

\begin{figure}
\begin{minipage}[t]{0.5\linewidth}
\centering
\includegraphics[width=0.8\textwidth]{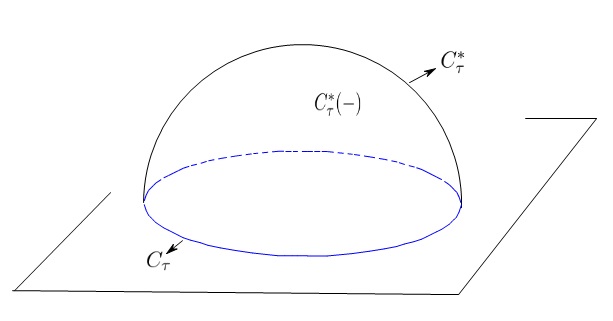}
\caption{A hyperbolic plane}\label{fig-g6}
\end{minipage}
\begin{minipage}[t]{0.5\linewidth}
\centering
\includegraphics[width=0.8\textwidth]{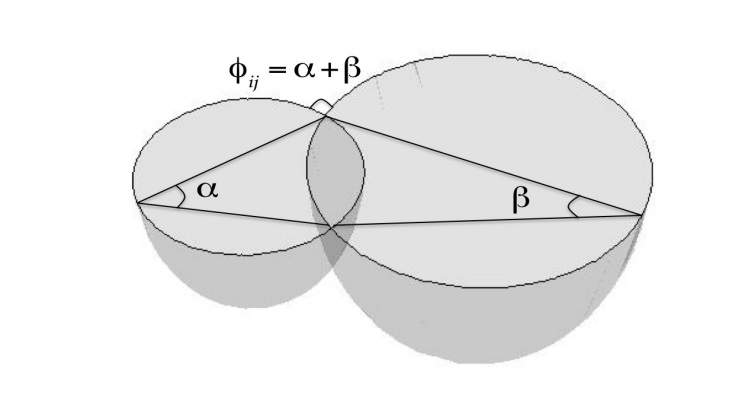}
\caption{Two intersecting hyperbolic planes}\label{fig-g5}
\end{minipage}
\end{figure}

\subsection{Delaunay triangulations and convex hyperbolic polyhedra}
The Delaunay condition ($\alpha_{ij}\leq \pi$ for each interior edge $ij$) can also be rephrased as ``the circumcircle of each triangle does not contain any vertices in its interior" \cite{Rivin-Annals}\cite{Springborn-JDG}. Given a locally finite Delaunay triangulation $T=(V,E,F)$ of $\mathbb{C}$. By definition, locally finite means that at each vertex $i\in V$, there is only finite vertices adjacent to $i$. We erase all such edge $ij$ with $\alpha_{ij}=\pi$, and obtain a reduced Delaunay decomposition $T^{red}=(V,E^{red},F^{red})$ of $\mathbb{C}$. Note that $E^{red}$ is a subset of $E$. Moreover, a face $\tau$ in the reduced decomposition $T^{red}$ may not be a triangle again. However, $\tau$ is always a finite convex polygon inscribed in a circle, which is denoted by $C_{\tau}$. Recall $\mathbb{C}$ is considered as the sphere at the infinity of the hyperbolic $3$-space $\mathbb{H}^3=\mathbb{C}\times\mathbb{R}_{>0}$. Thus $C_{\tau}$ is the boundary of a hyperbolic plane $C_{\tau}^*$ in $\mathbb{H}^3$, or say, $C_{\tau}$ is the intersection at infinity between $C_{\tau}^*$ and $\partial\mathbb{H}^3$. Geometrically, $C_{\tau}^*$ is the convex hull spanned by $C_{\tau}$ in $\mathbb{H}^3$. Obviously, the half sphere $C_{\tau}^*$ divide $\mathbb{H}^3$ into two part. Denote $C_{\tau}^*(-)$ by the open set in $\mathbb{H}^3$ below the half sphere $C_{\tau}^*$ and above the plane $\mathbb{C}$. See Figure \ref{fig-g6}. Then we obtain an ideal hyperbolic polyhedra with infinite vertices
$$\mathcal{P}(T,l)=\bigcap_{\tau\in F^{red}}\mathbb{H}^3\setminus C_{\tau}^*(-).$$
$\mathcal{P}(T,l)$ is convex, since $T$ is Delaunay. By definition, the dihedral angle of $\mathcal{P}(T,l)$ at an edge $ij\in E^{red}$ is the intersection angle between the two half spheres $C_{ijk}^*$ and $C_{ijl}^*$ (we assume that the two triangles $ijk$ and $ijl$ have a common edge $ij$, and are embedded in $\mathbb{C}$), which equals to the intersection angle $\Phi_{ij}$ between the two circles $C_{ijk}$ and $C_{ijl}$. By elementary arguments (or see \cite{Rivin}), one obtain $\Phi_{ij}=\alpha_{ij}$. See Figure \ref{fig-g5}. Thus the Delaunay condition $\alpha_{ij}\leq\pi$ says that all dihedral angles of $\mathcal{P}(T,l)$ are no more than $\pi$, which implies that $\mathcal{P}(T,l)$ is convex.

\subsection{A hyperbolic geometry interpretation of Theorem \ref{main}}
Let $(T_{hex}, l)$ be the standard hexagonal triangulation on $\mathbb{C}$ equipped with a PL-metric $l$. We assume that $(T_{hex}, l)$ is flat, complete and Delaunay. Recall $\mathcal{P}(T_{hex},l)$ is the corresponding ideal hyperbolic polyhedron constructed in the previous section. Its boundary $\partial\mathcal{P}(T_{hex},l)$ is a hyperbolic surface with infinite cusps $V$. By Theorem \ref{thm-BPS}, two such PL metrics $\tilde{l}$ and $l$ are PL conformal if and only if the hyperbolic surfaces $\partial\mathcal{P}(T_{hex},\tilde{l})$ and $\partial\mathcal{P}(T_{hex},l)$ are isometric. Thus Theorem \ref{main} may be rephrased as
\begin{theorem}
If the hyperbolic surface $\partial\mathcal{P}(T_{hex},l)$ with cusps is isometric to $\partial\mathcal{P}(T_{hex},l_0)$, where $l$ is flat, complete and Delaunay. Then the ideal polyhedron $\mathcal{P}(T_{hex},l)$ is isometric to $\mathcal{P}(T_{hex},l_0)$.
\end{theorem}

A convex ideal hyperbolic polyhedron $\mathcal{P}$ with infinite but locally finite faces is called \emph{hexagonally triangulated}, if the combinatoric of its boundary is equivalent to some reduced Delaunay decomposition $T_{hex}^{red}$ of $(T_{hex}, l_0)$. In other words, $\mathcal{P}$ is called hexagonally triangulated, if one can further triangulate its boundary (without adding new vertices) so as each vertex have valent six. In this case,
the combinatoric of the further triangulated boundary becomes equivalent to a hexagonal triangulation of $\mathbb{C}$. See Figure \ref{fig-g7}.

\begin{figure}[htbp]\centering
\includegraphics[width=0.5\textwidth]{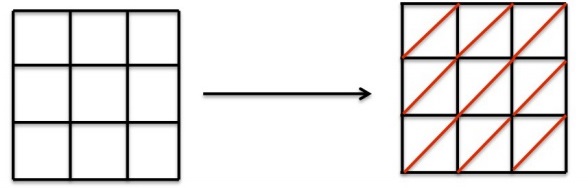}
\caption{Hexagonal triangulation}\label{fig-g7}
\end{figure}
\begin{corollary}
Given an infinite convex ideal hexagonally triangulated polyhedron $\mathcal{P}$ in $\mathbb{H}^3$. If $\partial\mathcal{P}$ is isometric to $\partial\mathcal{P}(T_{hex},l_0)$, then $\mathcal{P}$ is congruent to the standard ideal polyhedron $\mathcal{P}(T_{hex},l)$.
\end{corollary}

We refer to Luo \cite{Luo-book}, Rivin \cite{Rivin} and Springborn \cite{S} for more interpretations.

\noindent Song Dai, song.dai@tju.edu.cn\\[1pt]
\emph{Center for Applied Mathematics, Tianjin University, Tianjin, 300072, P.R. China}\\[1pt]

\noindent Huabin Ge, hbge@ruc.edu.cn\\[1pt]
\emph{School of Mathematics, Renmin University of China, Beijing, 100872, P.R. China}\\[1pt]

\noindent Shiguang Ma, msgdyx8741@nankai.edu.cn\\[1pt]
\emph{Department of Mathematics and LPMC, Nankai University, Tianjin, 300071, P.R. China}
\end{document}